\documentclass[12pt,reqno]{amsart}

\usepackage[utf8]{inputenc}
\usepackage{amssymb}
\usepackage{latexsym}
\usepackage{amsmath, amscd}
\usepackage[alphabetic]{amsrefs}

\usepackage{mathrsfs}
\usepackage{dsfont}
\usepackage{marvosym}

\usepackage{graphicx}
\usepackage{graphics}
\usepackage{subcaption}
\usepackage{wrapfig}

\usepackage{tikz}
\usepackage[all]{xy}

\usepackage{enumerate}
\usepackage{enumitem}
\usepackage{arydshln}

\usepackage{color}
\usepackage{xcolor}
\usepackage{verbatim}
\usepackage{soul}
\usepackage{ulem}

\usepackage[margin=1in]{geometry}

\usepackage{hyperref}
\definecolor{rose}{rgb}{0.82, 0.1, 0.26}
\definecolor{blu}{rgb}{0.36, 0.54, 0.66}
\definecolor{mor}{rgb}{0.55, 0.0, 0.55}

\hypersetup{
  colorlinks=true,
  linkcolor=rose,
  citecolor=mor,
  filecolor=blue
}

\usepackage{cleveref}

\theoremstyle{plain}
\newtheorem{theorem}{Theorem}[section]

\newtheorem{cor}[theorem]{Corollary}
\newtheorem{lem}[theorem]{Lemma}
\newtheorem{prop}[theorem]{Proposition}

\theoremstyle{definition}

\newtheorem{rmk}[theorem]{Remark}

\theoremstyle{remark}

\theoremstyle{definition}

\newcommand{\ZZ}{\mathds{Z}}
\newcommand{\CC}{\mathds{C}}

\newcommand{\NN}{\mathds{N}}
\newcommand{\PP}{\mathds{P}}
\newcommand{\PPP}{\PP^{3}}

\newcommand{\ii}{\mathscr{I}}
\newcommand{\oo}{\mathcal{O}}

\title{Curves contained in a quartic determinantal surface containing a line}

\author{Abel Castorena}
\address{Centro de Ciencias Matemáticas, Universidad Nacional Autónoma de México, Morelia, Michoacán}
\email{abel@matmor.unam.mx}

\author{Monserrat Vite}
\address{Centro de Ciencias Matemáticas, Universidad Nacional Autónoma de México, Morelia, Michoacán}
\email{montserrat@matmor.unam.mx}

\begin{document}

\maketitle

\begin{abstract}
 Let $X\subseteq \PPP$ be a very general element of the Noether–Lefschetz divisor that parametrizing smooth quartic surfaces containing a line. Let $L\subseteq X$ denote the corresponding line. We study the curves contained in $X$ and analyze their behavior in the Hilbert scheme. We first determine which linear systems contain smooth irreducible curves. For most classes, we verify that the general member is a smooth point of the expected Hilbert scheme. Finally we compute the Rao function of any curve on $X$.
\end{abstract}

\section{Introduction}

The geometry of curves on K3 surfaces is a central theme in algebraic geometry, owing to the rich structure of the linear systems defined upon them, the deformation theory of embedded curves, and the geometry of components of the Hilbert scheme parametrizing curves in $\PPP$. Among K3 surfaces, smooth quartics in $\PPP$ are the simplest examples accessible to explicit computation, making them an ideal setting
for studying the geometric properties of curves contained in them.

A natural first step in understanding the curves on such a surface is to determine which of them are arithmetically Cohen-Macaulay (ACM). In \cite{CV26}, the authors give a complete classification of ACM curves on a surface of degree $d$ in $\PPP$ in terms of weak admissible pairs, and apply this classification to describe explicitly
all ACM curves contained in a very general smooth determinantal quartic surface $X \subseteq \PPP$. In particular, when $X$ contains a line $L$, so that $\mathrm{Pic}(X)$ is generated by the hyperplane class $H$ and the line class $L$, the classification of \cite{CV26} shows that a curve with class $aH+bL$ on $X$ is ACM if and only if $b\in\{0,1,-1\}$. This result serves as a starting point and primary motivation for the present work: our goal is to go beyond the ACM locus and study \emph{all} curves contained in $X$, whether ACM or not.

In this paper we focus on a smooth quartic surface $X \subseteq \PPP$ containing a line $L$, so that $\mathrm{Pic}(X)$ is generated by $H$ and $L$ with $H^2=4$, $H\cdot L=1$, and $L^2=-2$. The purpose of this article is threefold.

The first goal is to determine in which linear systems $|(k+m)H-mL|$ and $|kH+mL|$ there exist smooth irreducible curves. These classes represent all possible effective curves on $X$: by varying $k$ and $m$ over the natural numbers, one obtains the effective cone of $X$. To distinguish classes with smooth members from those whose general element is necessarily reducible or degenerate, we rely on a criterion due to Kleppe \cite{kleppe1}. Specifically, we apply \cite[Lema 3.1]{kleppe1} to show that smooth irreducible curves exist in the classes $|(k+m)H-mL|$ for $k,m\geq 1$, and in $|kH+mL|$ for $m\geq 1$ and $k\geq 2m$. On the other hand, the classes $|(2m-1)H+mL|$ are shown to have only reducible general members, consisting of $L$ together with a smooth curve in $|(2m-1)H+(m-1)L|$. %and the general member of $|mH-mL|$ is a union of $m$ plane cubics contained in $X$.

The second goal is to understand the position of these curves in the Hilbert scheme $H_{d,g}$ of curves of degree $d$ and genus $g$ in $\PPP$. We describe the behavior of curves on $X$ inside their corresponding Hilbert schemes, distinguishing two principal families of divisor classes.In the linear system $|(k+m)H-mL|$ we proved that the general element is a smooth point of the Hilbert scheme for all $m\geq 1$ and $0<k\not=4$. In the case of curves in the linear system $|(4+m)H-mL|$ the points are singular points. For curves in the linear systems $|kH+ mL|$, we verify that when $k > 2m+2$ and $m \geq 1$, the general member is a smooth point of a component of the Hilbert scheme with the expected dimension and for the cases $k\in{\{2m,2m+1,2m+2\}}$ the points are in a non-reduced component of the correspondent Hilbert scheme. The remaining cases are $kH+mL$ with $m\geq 2$ and $k\in\{1, \ldots ,2m-1\}$. We know that these divisors are reducible; however, in general, we do not know how they decompose. To gain some insight, we computed several small cases using the software Macaulay2.

The third goal is to compute the Rao function of every curve on $X$, obtaining
explicit formulas that describe their cohomological behavior in terms of the class in
$\mathrm{Pic}(X)$. Additionally, we explain how to calculate the minimal resolution of any curve in $X$, which allows us to calculate its postulate character $\gamma$ and thus have information about the scheme with constant cohomology $H_{\gamma,\rho}$ corresponding to the curve, which stratifies the Hilbert scheme $H_{d,g}$ where the curve lies (see \cite{mdp2}).
\vskip2mm

\noindent\textbf{Acknowledgements.} The second author thanks Manuel Leal for his support with the more technical details of this writing. The first author is supported with research grant PAPIIT IN101226 ``Teoría de Brill-Noether y aplicaciones" from UNAM, M\'exico. Second author is a SECIHTI fellow at Centro de Ciencias Matemáticas (UNAM), México.

\medskip
\noindent\textbf{Organization of the paper.}
In Section~\ref{sec:smooth} we analyze the effective cone and determine the classes
with smooth irreducible members. In Section~\ref{sec:hilbert} we study the Hilbert
schemes of these curves. In Section~\ref{sec:rao} we compute the Rao function. An
appendix provides the \textsc{Macaulay2} code used in our computations.
 
\section{Smooth curves in a Picard class} \label{sec:smooth}
Let $X=V(F)\subseteq \PPP$ be a very general smooth surface of degree $4$ containing a line $\ell$. The purpose of these notes is to obtain information of the curves that are contained in $X$. We know that the Picard group of $X$ is of rank $2$ and that the line $\ell$ is the unique line contained in $X$. We denote the class of the line $\ell$ by $L$, then $Pic(X)$ is generated by $H$ and $L$ with $H^{2}=4, H.L=1$ and $L^{2}=-2$.\\

The effective cone of $X$ is generated by the classes $H-L$ and $L$, this implies that an effective curve has class $aH+(b-a)L$ with $a,b\in{\mathds{N}}$. Then we have the following corollary from the Proposition 19 of \cite{CV26}:

\begin{cor} \label{coro:acm}
    Let $C$ be a curve in $X$ with class $aH+(b-a)L$, then $C$ is an ACM curve if and only if $b-a\in{\{0,1,-1\}}$.
\end{cor}

In order to know which classes have smooth irreducible curves we use a result of \cite{kleppe1} that in our notation says:
\begin{lem}[\cite{kleppe1}*{Lem. 3.1}] \label{lemkleepe}
Let $D$ be an effective divisor on $X$. If $D.L\geq 0$ and $D.(H-L)>0$, then the linear system $|D|$ contains smooth irreducible curves.    
\end{lem}

 A particular case of Lemma 2.2 is the following corollary;
\begin{cor} \label{cor:irreduciblecurves}
    \begin{enumerate}
        \item Let $m\geq 1$ and $k\geq 1$, then $|(k+m)H-mL|$ contains a smooth irreducible curve.
        \item Let $m\geq 1$ and $k\geq 2m$, then $|kH+mL|$ contains a smooth irreducible curve.
    \end{enumerate}
\end{cor}
\begin{comment}
\begin{proof}
     By Lemma \ref{lemkleepe} is enough with intersect the class with $L$ and $H-L$. For (1), we have:
  $$((k+m)H-mL).L=k+3m\geq 0 \, \text{ and }\, ((k+m)H-mL).(H-L)=3k+4m>0$$
  For the item (2):
  $$(kH+mL).L=k-2m\geq 0 \, \text{ and }\, (kH+mL).(H-L)=3k+3m>0$$
\end{proof}
\end{comment}

\begin{prop} Let $m\geq 1$, then:
    \begin{enumerate}
        \item For $k\in \{1,\ldots ,2m-2\}$, the linear system $|kH+mL|$ does not contain smooth irreducible curves.
        \item The linear system $|(2m-1)H+mL|$ does not contain smooth irreducible curves. Furthermore, the general element of the linear system is the union of $L$ with a smooth curve in $|(2m-1)H+(m-1)L|$.
        \item The general member in the linear system $|mH-mL|$ is the union of $m$ plane cubics contained in $X$.
    \end{enumerate}
\end{prop}
\begin{proof}
    For the first item we have that it is enough to prove that $h^{1}(X,\oo_{X}(kH+mL))\not=0$ by \cite{k3sup}*{Rmk. 2.2}. Let $m\geq 1$ and $k\in \{1,\ldots ,2m-2\}$, then we have the following exact sequence
    \begin{equation} \label{sucexact3}
        0\to \oo_{X}(kH+(m-1)L)\to \oo_{X}(kH+mL)\to \oo_{L}(kH+mL) \to 0
    \end{equation}
Since the degree of the divisor $kH+mL$ on $L$ is negative, $deg_{L}(kH+mL)=k-2m<0$ we have that $H^{0}(L,\oo_{L}(kH+mL))=0$. On the other hand, the divisor $-(kH+mL)$ is not effective, then $H^{2}(X,\oo_{X}(kH+mL))\cong H^{0}(X,\oo_{X}(-(kH+mL)))=0$, therefore, by the long exact sequence in cohomology associated to \eqref{sucexact3} we have that
\begin{align*}
    h^{1}(X,\oo_{X}(kH+mL))&=h^{1}(X,\oo_{X}(kH+(m-1)L))+h^{1}(X,\oo_{L}(kH+mL)) \\
    &=h^{1}(X,\oo_{X}(kH+(m-1)L))+h^{0}(\PP^2 ,\oo_{\PP^{2}}(-2-deg_{L}(kH+mL))) \\
    &=h^{1}(X,\oo_{X}(kH+(m-1)L))+2m-k-2+1>0.
\end{align*}

In order to prove the second item, consider the linear subsystem $$A=\{C\cup L|C\in|(2m-1)H+(m-1)L|\}\subseteq |(2m-1)H+mL|.$$ Observe that $(2m-1)H+(m-1)L=(2m^{\prime}+1)H+m^{\prime}L$ with $2m^{\prime}+1\geq 2m^{\prime}$, then by Lemma \ref{cor:irreduciblecurves} there exists a smooth irreducible curve $C\in|(2m-1)H+(m-1)L|$ and by \cite{k3sup}*{Lem. 2.1} we have that $h^{0}(X,\oo_{X}((2m-1)H+(m-1)L))=g(C)+1$, then the linear subsystem $A$ has dimension $g(C)$, thus the complete linear system has dimension at least $g(C)$. The genus of $C$ is $$g(C)=2(2m-1)^{2}+(2m-1)(m-1)-(m-1)^{2}+1=2(2m-1)^{2}+(2m-1)m-m^{2}+1$$

    On the other hand, given $D\in|(2m-1)H+mL|$, by Riemann-Roch we know that $$\chi (X,(2m-1)H+ml)=g(D)+1$$
    thus $h^{0}(X,\oo_{X}((2m-1)H+mL))\leq g(D)+1$ and therefore, $dim |(2m-1)H+ml|\leq g(D)$. But $g(D)=2(2m-1)^{2}+(2m-1)m-m^{2}+1$. This implies that $A$ is the complete linear system, then the generic element of the linear system is reducible.

    Finally, observe that $$g(D)=g(C\cup L)=g(C)+g(l)+\# (C\cap L)-1$$
    since $g(D)=g(C)$ we conclude that $\#(C\cap L)=1$ and therefore the general element is singular.\\
    For the last item, let $H_{1}$ be a plane that contain $L$, thus $C:=H_{1}\cap X-L$ is a curve with class $H-L$, then $H-L$ is a movible class, therefore has not fixed components, then by \cite{sdonat}*{Prop. 2.6} we obtain the result.\\
\end{proof}

\begin{cor} \label{cor:cohmh-ml}
    For all $m\geq 1$ we have that $h^{0}(X,\oo_{X}(mH-mL))=m+1.$
\end{cor}
\begin{proof}
    By the last proposition we have that the class $H-L$ has not fixed component and $(H-L)^{2}=0$, then by \cite{sdonat}*{Prop. 2.6} we have that $h^{1}(X,\oo_{X}(mH-mL))=m-1$. Since $h^2(X,\oo_{X}(mH-ml))=h^{0}(X,\oo_{X}(mL-mH))=0$ we have that:
{\small    $$h^{0}(X,\oo_{X}(mHmL))=\chi (X,mH-ML)+h^1(X,\oo_{X}(mH-mL))=2+\frac{(mh-mL)^{2}}{2}+m-1=m+1.$$}
\end{proof}

\section{Curves and their Hilbert scheme} \label{sec:hilbert}
 A natural question is if the curves contained in $X$ are smooth points in their corresponding Hilbert scheme. In this section we answer this question for almost all curves. \\
 
 The following two families of curves are used to find elements in any class:
\begin{itemize}
    \item Let $F_{m}$ the curve defined by the ideal $I_{F_{m}}=I_{\ell}^{m}+\langle F \rangle$. By definition this curve has class $mL$ on $Pic(X)$, and therefore has degree $m$ and genus $-m^{2}+1$. The minimal free resolution of the ideal of curve $F_{m}$ is:
\begin{equation} \label{resmL}
    0\to \oo (-(m+4))^{m-1}\to \oo (-(m+3))^{m}\oplus \oo (-(m+1))^{m}\to \oo (-4)\oplus \oo (-m)^{m+1}\to I_{F_{m}} \to 0
\end{equation}

\item Let $G_{m}$ be the disjoint union of $m$ plane curves of degree $3$ contained in $X$. The class of $G_{m}$ is $mH-mL$ and has degree $3m$ and genus $1$. This curve is linked to a curve in the class $mL$ by a complete intersection $4\times m$. The minimal free resolution of the ideal of a curve $G_{m}$ is
\begin{multline}
    \label{resmH-mL}
    0\to \bigoplus_{j=0}^{m-2}\oo (-(3m+2-2j))\to \bigoplus_{j=0}^{m-1}\oo (-(3m+1-2j))^{2}\\
    \to \oo (-4)\bigoplus_{j=0}^{m}\oo (-(3m-2j))\to I_{G_{m}} \to 0
\end{multline}

\end{itemize}

\subsection{The class $(k+m)H-mL$} $ $\\

Curves $C_{k}$ in the linear system $|(k+m)H-mL|$ have degree $d_{k}=4k+3m$ and genus $g_{k}=2k^{2}+3mk+1$. These curves are linked to a curve $F_{m}\in{|mL|}$ by the complete intersection of $X$ with a surface of degree $(k+m)$. And therefore, a curve $C_{k+1}$ is obtained from a curve $C_{k}$ by an elementary biliaison $(4,1)$ (see \cite{mdp2}*{III Def. 2.1}). This allows us to compute the minimal free resolution of any element from the minimal free resolution of $C_{1}$. \\

For example, for $m=2$, we have that $C_{1}\in{|3H-2L|}$ is a curve of degree $10$ and genus $9$ with minimal free resolution:
\begin{multline*}\to \oo_{\PPP}(-9)\to \oo_{\PPP}(-8)^{2}\oplus \oo_{\PPP}(-6)^{2}\\
\to \oo_{\PPP}(-4)\oplus \oo_{\PPP}(-3)\oplus \oo_{\PPP}(-5)\oplus \oo_{\PPP}(-7)\to I_{C_{1}} \to 0
\end{multline*}
Then, by \cite{mdp2}*{Prop. 4.3} the minimal free resolution of $C_{k}\in{|(k+2)H-2L|}$ is:
 \begin{multline*}0\to \oo_{\PPP}(-(8+k))\to \oo_{\PPP}(-(7+k))^{2}\oplus \oo_{\PPP}(-(5+k))^{2} \\
\to \oo_{\PPP}(-4)\oplus \oo_{\PPP}(-(k+2))\oplus \oo_{\PPP}(-(k+4))\oplus \oo_{\PPP}(-(k+6)) \to I_{C} \to 0\end{multline*}

On the other hand, we can say something about Hilbert's scheme where these curves reside. First we denote by $H_{d,g}$ the Hilbert scheme of curves of degree $d$ and genus $g$ in $\PPP$.

\begin{prop}
    Let $|(k+m)H-mL|$, then:
    \begin{enumerate}
        \item For $k\geq 5$ and $m\geq 1$,the general member of this class (which is smooth and irreducible by Corollary \ref{cor:irreduciblecurves}) is contained in a generically smooth component $W$ of the Hilbert scheme $H_{4k+3m,2k^{2}+3mk+1}$. Furthermore, the dimension of $W$ is $2k^{2}+mk+34$. 
        \item For $k\in{\{1,2,3\}}$ and $m\geq 1$ the general element in $|(k+m)H-mL|$ is a smooth point of a component $W$ of the Hilbert scheme $H_{4k+3m,2k^{2}+3mk+1}$ of dimension $4(4k+3m)$.
        \item In the class $|(4+m)H-mL|$ with $m\geq 2$, the general element is a singular point of the Hilbert Scheme $H_{16+3m,33+12m}$.
    \end{enumerate}
    \end{prop}
    \begin{proof}
    Observe that $L$ and $H-L$ is another basis of $Pic(X).$
 \item For the item (1), let $k\geq 5$ and $m\geq 1$ we rewrite the class $(k+m)H-mL$ using the new basis as $kL+(k+m)(H-L)$. Since $4<k<\frac{3(k+m)-2}{2}$ for all $k\geq 5$, the result follows from \cite{kleppe1}*{Thm. 3.4}.

    \item For the second item, let $k\in{\{1,2,3\}}$ and $m\geq 1$, a general element $C$ in the class $(k+m)H-mL$, it is a smooth and irreducible curve by Corollary  \ref{cor:irreduciblecurves} and by \cite{kleppe2}*{Lem. 13} and \cite{kleppe3}*{Rmk. 1.13} we have that there exist an injective map $H^{1}N_{C}\to H^{1}\oo_{C}(4)$.  But from the long exact sequence on cohomology associated to the exact sequence of $C$ on $X$ twisted by $4H$ we have that 
    $$0= H^{1}(X,\oo_{X}(4H))\to H^{1} (C, \oo_{C}(4)) \to H^{2}(X,\oo_{X}(4H-C))\cong H^{0}(X,\oo_{X}(C-4H))=0.$$
    Then $h^{1}N_{C}\leq h^1 \oo_{C}(4)=0$, thus $C$ is a smooth point of the Hilbert scheme and since $\chi (N_{C})=h^{0}N_{C}-h^{1}N_{C}=4d_{C}$, we have that $C$ is contained in a component of dimension $4d_{C}=4(4k+3m)$.

    \item Finally, for the third item, suppose that $m\geq 2$ and let $C$ be a general element in the class $(4+m)H-mL$ (which is smooth and irreducible by Corollary \ref{cor:irreduciblecurves}). As before, by \cite{kleppe2}*{Lem. 13} and \cite{kleppe3}*{Rmk. 1.13} there exist an injective map $H^{1}N_{C}\to H^{1}\oo_{C}(4)$. Then $h^{1}N_{C}\leq h^{1}\oo_{C}(4)$. On the other hand, since $C\subseteq X \subseteq \PPP$, we obtain the following exact sequence of normal bundles:
    \begin{equation}
        0\to N_{C|X} \to N_{C} \to N_{X|\PPP_{| C}} \to 0
    \end{equation}
    This induces the following long exact sequence in cohomology:
    $$\ldots \to H^{1}N_{C} \to H^{1}N_{X|\PPP_{| C}}\cong H^{1}\oo_{C}(4)\to H^{2}N_{C|X}=0 \to \ldots $$
    Whence $h^{1}\oo_{C}(4)\leq h^{1}N_{C}$, thus $h^1 N_{C}=h^{1}\oo_{C}(4)$. But from the long exact sequence on cohomology associated to the exact sequence of $C$ on $X$ twisted by $4H$ we have that 
    $$0 \to H^{1} (C, \oo_{C}(4)) \to H^{2}(X,\oo_{X}(4H-C))\cong H^{0}(X,\oo_{X}(C-4H))\to 0.$$
    Since $C-4H=mH-mL$, by Corollary \ref{cor:cohmh-ml} we have that $h^{1}N_{C}=h^{1}\oo_{C}(4)=m+1$. But $\chi (N_{C})= 4d(C)$, then $h^{0}N_{C}=4(16+3m)+m+1$.\\
    We consider the incidence variety:
    $$\xymatrix{ & \{(C,X)\in H_{16+3m,33+12m}\times det(a,b) |C\subseteq X\} \ar[dl]_{\pi_{0}} \ar[dr]^{\pi_{1}} & \\
    \mathcal{B}\subseteq H_{16+3m, 33+12m} & & det(a,b)}$$
    Where $det(a,b)$ is the family of determinantal surfaces of type $((1,1),(2,4))$ (See \cite{LLV}*{Def. 2}). We restrict ourselves to this variety because, by hypothesis, the curves $C$ that we consider are contained in $X$, which is a very general element of $det(a,b)$. We want to compute the dimension of the component $\mathcal{B}$, for this, first observe that:
    \begin{enumerate}
        \item Since the degree of a curve $C\in\mathcal{B}$ is $16+3m>16$ and $X$ is smooth irreducible quartic surface, $C$ cannot be contained in more than two quartics then the dimension of $\pi_{0}^{-1}$ is zero. 
        \item By \cite{LLV}*{Thm. 1} the dimension of $det(a,b)$ is 33.
        \item By definition, $dim(\pi_{1}^{-1})=dim|C|=h^0 (X,\oo_{X}(C))-1=g(C)$, where the last equality follows form  \cite{k3sup}*{Lem 2.1}.
    \end{enumerate}
    Therefore, 
    \begin{align*}
        dim\,\mathcal{B}&=dim\,det(a,b)+dim(\pi_{1}^{-1})-dim(\pi_{0}^{-1})=33+g(C)=33+33+12m=2(33+6m)\\
        &=4(16+3m)+2<4(16+3m)+m+1
    \end{align*}
    then $C$ is a singular point of the Hilbert scheme.\\
\end{proof}

\subsection{The class $kH+mL$} $ $\\

Now consider curves $C_{k}$ in the linear system $|kH+mL|$, these curves have degree $d_{k}=4k+m$ and genus $g_{k}=2k^{2}+mk-m^{2}+1$. As before, the curves are linked by the complete intersection of $X$ with a surface of degree $(k+m)$  to a curve $G_{m}\in{|mH-mL|}$. And therefore, a curve $C_{k+1}$ is obtained from a curve $C_{k}$ by an elementary biliaison $(4,1)$. This allows us to compute the minimal free resolution of any element from the minimal free resolution of $C_{1}$. \\

When $k\geq 2m$, the general member on the linear system $|kH+mL|$ is smooth and irreducible and we can say something about their Hilbert scheme.

\begin{prop} \label{prop3.2}
    Let $m\geq 2$ and $C\in{|kH+mL|}$ be a general element, then
    \begin{enumerate}
        \item If $k>2m+2$, then $C$ is a smooth point of a component $W$ of the Hilbert scheme $H_{4k+3m,2k^{2}+mk-m^{2}+1}$. Furthermore, the dimension of $W$ is $2k^{2}+mk-m^{2}+34$.
        \item If $k\in{\{2m,2m+1,2m+2\}}$ then $C$ is contained in a non-reduced component $W$ of the Hilbert scheme $H_{4k+3m,2k^{2}+mk-m^{2}+1}$. 
    \end{enumerate}
\end{prop}
\begin{proof}
    First we rewritten the class $kH+mL$ on the basis $L, H-L$ as $(k+m)L+k(H-L)$. Then the result it follows from \cite{kleppe1}*{Thm. 3.4}, the first item since $4<k+m<\frac{3k-2}{2}$ for all $k> 2m+2$, and the second item since when $k\in\{2m,2m+1,2m+2\}$, we have that $\frac{3k-2}{2}\leq k+m \leq \frac{3k}{2}$.\\
\end{proof}

We do not have a similar result as Proposition \ref{prop3.2} for every curve in $|kH+mL|$, but for some cases, as shown in the following lemma.

\begin{lem}
    Set $m\geq 2$. Let $C$ be a general element on the linear system $|H+mL|$ then $C$ is the union of a curve $F_{m}\in |mL|$ and a plane curve of degree $4$ contained in $X$.
\end{lem}
\begin{proof}
    Let $C$ be a general member of $|H+mL|$, since $H+mL=(m+1)H-(mH-ML)$ we have that the curve $C$ is linked to a curve $G_{m}\in{|mH-mL|}$ by the complete intersection of $X$ with a surface $S$ of degree $m+1$. From the resolution \eqref{resmH-mL} we conclude that the degrees of a minimal set of generators of the ideal of $G_{m}$ are $4,3m-2,3m-4, \ldots ,m+2, m$, therefore the surfaces of degree $m+1$ containing $G_{m}$ have as factor the quartic equation of $X$ or a equation of degree $m$ product a line equation. Since $X$ and $S$ is not have common factors $S$ must be of the second type, then $S$ is the union of a plane $H_{1}$ with a surface $S^{\prime}$ of degree $m$. Since $G_{m}$ is the union of $m$ plane cubics with $m>1$, $G_{m}$ does not contained in $H_{1}\cap X$ and therefore a component of $C$ is contained in $H_{1}\cap X$ which is a plane quartic curve $A$ with class $H$, therefore $H+mL=[C]=[A]+ [B]$, then $B\in{mL}$.\\
\end{proof}

Studying the missing classes is trickier. We present next geometric properties for some cases.
To get an idea of what the missing classes look like, we used the Macaulay2 program \cite{m2}. 

\begin{itemize}
    \item The class $kH+2L$:
     \begin{itemize}
        \item[\textbf{$k=1$:}] In this case we have a curve $C$ of degree $6$ and genus $1$ which is the union of a curve $F_{2}$ and a plane quartic. The family of curves of this kind has dimension $26$. Using Macaulay2 we compute $h^{0}(C,N_{C})=26$ and $h^{1}(C,N_{C})=2$, therefore $C$ is a smooth point in a component the Hilbert scheme $H_{6,1}$.
        \item[\textbf{$k=2$:}] The curves in this class have degree $10$ and genus $9$. A general member $C$ is the union of the line $\ell$ and a curve $B$ smooth of degree 9 and genus 10 which is a complete intersection of two surfaces of degree $3$. The dimension of the family of curves of this kind is $40$ and Macaulay2 compute $h^{0}(C,N_{C})=40$ and $h^{1}(C,N_{C})=0$. Then, $C$ is a smooth point of a component of the Hilbert scheme $H_{10,9}$.
        \item[\textbf{$k=3$:}] For this case we have curves of degree 14 and genus 21 and the general member is the union of the line $\ell$ and a smooth ACM curve of degree 13 and genus 21 with minimal free resolution:
    $$0\to \oo_{\PPP}(-7)\oplus \oo_{\PPP}(-5)\to \oo_{\PPP}(-4)^{3}\to \ii_{B}\to 0 $$
In this case the family of this curves  has dimension $56$ and Macaulay gives $h^{0}(C,N_{C})=56$ and $h^{1}(C,N_{C})=0$, thus $C$ is a smooth point in a component of the Hilbert scheme $H_{14,21}$.
\end{itemize}

\item The class $kH+3L$:
     \begin{itemize}
        \item[\textbf{$k=1$:}] In this case we have a curve $C$ of degree $7$ and genus $-3$ which decomposes as the union of a curve $F_{3}$ and a plane quartic. The family of curves of this kind is a family of dimension $32$. With Macaulay2 we compute $h^{0}(C,N_{C})=34$ and $h^{1}(C,N_{C})=6$. In this case we cannot conclude the smoothness in the Hilbert scheme. 
        \item[\textbf{$k=2$:}] The curves in the class $|2H+2L|$ have degree $11$ and genus $6$. A general member $C$ is the union of a curve $F_{2}$ and a curve $B$ smooth of degree 9 and genus 10 which is a complete intersection of two surfaces of degree $3$. The dimension of the family of curves of this kind is $47$ and Macaulay2 compute $h^{0}(C,N_{C})=47$ and $h^{1}(C,N_{C})=3$. Then, $C$ is a smooth point of a component of the Hilbert scheme $H_{10,9}$.
        \item[\textbf{$k=3$:}] For this case we have curves of degree 15 and genus 19 and the general member is the union of a curve $F_{2}$ and a curve $B$ which is a smooth ACM curve of degree 13 and genus 21 with minimal free resolution:
    $$0\to \oo_{\PPP}(-7)\oplus \oo_{\PPP}(-5)\to \oo_{\PPP}(-4)^{3}\to \ii_{B}\to 0 $$
In this case the family of this curves  has dimension $53$ and Macaulay gives $h^{0}(C,N_{C})=62$ and $h^{1}(C,N_{C})=2$, thus we cannot conclude the smoothness.
\end{itemize}

\item The class $kH+4L$:
     \begin{itemize}
        \item[\textbf{$k=1$:}] For this we have a curve $C$ of degree $8$ and genus $-9$ which decomposes as the unión of a curve $F_{4}$ and a plane quartic $A$. The family of curves of this kind is a family of dimension $36$. Using Macaulay2 we compute $h^{0}(C,N_{C})=44$ and $h^{1}(C,N_{C})=12$, therefore we cannot conclude that $C$ is a smooth point or not.
\end{itemize}
\end{itemize}

\section{The Rao funtion of a curve in $X$} \label{sec:rao}
We define the Rao module of a curve in $\PPP$ as:
$$M(C)= \bigoplus_{n\in{\mathds{Z}}} M(C)_{n}= \bigoplus_{n\in{\mathds{Z}}} H^{1}(\PPP,I_{C}(n))$$

 We are interested in computing the Rao function of a curve $C$ in $X$, that means, the funtion $\rho_{C}:\ZZ \to \NN$ defined by $n\mapsto dim_{\CC}\,M(C)_{n}$; for this reason, we first compute the Rao function of a curve $F_{m}$. The case $m=1$, $mL$ is the line which it has as Rao function the constant zero function. The case $m=2$ is in \cite{mig2line}*{Lemma 1.2}. Therefore, we focus in the cases $m\geq 3$. 
 
 \begin{rmk} \label{rmk:igualdad}
  From the exact sequence:
    $$0\to I_{C}(n) \to \oo_{\PPP}(n) \to \oo_{C}(n)\to 0$$
    We get the associated long exact sequence in cohomology:
      \begin{equation}
     0\to H^{0}(\PPP,I_{C}(n))\to H^{0}(\PPP,\oo_{\PPP}(n)) \to H^{0}(C,\oo_{C}(n)) \to M(C)_{n} \to 0    
    \end{equation}
    Therefore, we get that:
    \begin{align*}
        dim_{\CC}M(C)_{n}&=h^{0}(C,\oo_{C}(n))-(h^{0}(\PPP,\oo_{\PPP}(n))-h^{0}(\PPP,I_{C}(n)))\\
        &=h^{0}(C,\oo_{C}(n))-dim_{\CC}(S/I_{C})_{n}
    \end{align*}
 \end{rmk}
 
 First we compute the dimension of the graded part of degree $n$ of the quotient $S/I_{F_{m}}$. We distinguish three cases:

\begin{itemize}
\item[$m=3$] from the minimal free resolution of $I_{F_{3}}$ we conclude that:

          \begin{equation} \label{dimS/I3}
          dim_{\CC}(S/I_{C})_{n}=\left\{
  \begin{array}{ll}
		\binom{n+3}{3}    & \mathrm{if} \; 0\leq n\leq 2; \\
	    16  & \mathrm{if} \; n =3 ;\\ 
         3n+9   & \mathrm{if} \; n \geq 4. 
		 \end{array}
	     \right.
      \end{equation}  
         \item[$m=4$] from the minimal free resolution of $I_{F_{4}}$ we conclude that:

        \begin{equation} \label{dimS/I4}
            dim_{\CC}(S/I_{F_{4}})_{n}=\left\{
  \begin{array}{ll}
		\binom{n+3}{3}    & \mathrm{if} \; 0\leq n\leq 3; \\
	    29  & \mathrm{if} \; n =4 ;\\ 
         4n+16   & \mathrm{if} \; n \geq 5. 
		 \end{array}
	     \right.
        \end{equation}

         \item[$m\geq 5$] from the minimal free resolution of $I_{F_{m}}$ we conclude that:
        \begin{equation} \label{dimS/Im}
            dim_{\CC}(S/I_{F_{m}})_{n}=\left\{
  \begin{array}{ll}
		\binom{n+3}{3}    & \mathrm{if} \; 0\leq n\leq 3; \\
	    \binom{n+3}{3}-\binom{n-1}{3}  & \mathrm{if} \; 4\leq n \leq m-1;\\
         \binom{n+3}{3}-\binom{n-1}{3} -n-1 & \mathrm{if} \; n=m \\
         mn+m^{2}   & \mathrm{if} \; n \geq m+1. 
		 \end{array}
	     \right.
        \end{equation}

\end{itemize}

Now we need to compute $h^{0}(C,\oo_{C}(n))$ for all $n$. For this, we use the following:

\begin{prop}
    For $m\geq 3$, let $C=F_{m}$, then we have that:
    $$h^{0}(C,\oo_{C}(n))=\left\{
  \begin{array}{ll}
        h^{1}(X,\oo_{X}(nH-mL))   & \mathrm{if} \;  n\leq 0; \\
        \binom{n+3}{3} -2-n^{2}+mn+m^{2}   & \mathrm{if} \; 0< n< 3; \\
        mn+m^{2}  & \mathrm{if} \; 3\leq n.\\ 
         \end{array}
    \right.$$
\end{prop}
\begin{proof}
    Observe that we have the following exact sequence:
         \begin{align} \label{exacseqK}
             \xymatrix{0 \ar[r]& \mathcal{K}\ar[r]& \oo_{C} \ar[r]^{\varphi}& \oo_{\ell} \ar[r] &0}
         \end{align}
         With $\mathcal{K}$ the kernel of the map $\varphi$. Note that:
         $$\mathcal{K}=I_{\ell/C}\cong I_{\ell /X}/I_{C/X}\cong \oo_{X}(-L)/\oo_{X}(-mL)$$
         Then, for all $n$ we have the following exact sequence:
         \begin{equation} \label{sucexcK}
             0\to \oo_{X}(nH-mL) \to \oo_{X}(nH-L) \to \mathcal{K}(n) \to 0
         \end{equation}
         and the associated long exact sequence in cohomology is:
         
    \begin{align*}
     0 \to H^{0}(X,\oo_{X}(nH-mL))\to H^{0}(X,\oo_{X}(nH-L)) \to H^{0}(X,\mathcal{K}(n)) \\
     \to H^{1}(X,\oo_{X}(nH-mL))  \to H^{1}(X,\oo_{X}(nH-L))\to H^{1}(X,\mathcal{K}(n))\\
     \to H^{2}(X,\oo_{X}(nH-mL))\to \ldots   
    \end{align*}
Since $L$ is an ACM curve, we have that $h^{1}(X,\oo_{X}(nH-L))=0$ for all $n$, on the other hand for $n>0$ the class $mL-nH$ is not effective thus $h^{2}(X,\oo_{X}(nH-mL))=h^{0}(X,\oo_{X}(mL-nH))=0$, consequently $h^{1}(X,\mathcal{K}(n))=0$ for all $n>0$ and we have that:
\begin{equation} \label{eqK}
    h^{0}(X,\mathcal{K}(n))=h^{1}(X,\oo_{X}(nH-3L))+h^{0}(X,\oo_{X}(nH-L))-h^{0}(X,\oo_{X}(nH-3L)) \, \forall n
\end{equation}
Then, by the associated long exact sequence in cohomology of \eqref{exacseqK} we have that:
\begin{equation} \label{eqOc}
    h^{0}(C,\oo_{C}(n))=h^{0}(\ell ,\oo_{\ell}(n))+h^{0}(X,\mathcal{K}(n)) \qquad \forall n>0.
\end{equation}

From the long exact sequence in cohomology associated to the exact sequence of the line in $X$, we get:
\begin{equation} \label{eqOL}
    h^{0}(\ell ,\oo_{\ell}(n))=h^{0}(X,\oo_{X}(nH))-h^{0}(X,\oo_{X}(nH-L)) \qquad \forall n
\end{equation}
Thus, combining the equations \eqref{eqK}, \eqref{eqOc} and \eqref{eqOL} we get that:
\begin{equation*}
    h^{0}(C,\oo_{C}(n))=\left\{
  \begin{array}{ll}
		\binom{n+3}{3}-\chi(nH-mL)=\binom{n+3}{3}-2-2n^{2}+nm+m^{2}   & \mathrm{if} \; 0< n\leq 2; \\
	     \binom{n+3}{3}-\binom{n-1}{3}-\chi (nH-mL)=nm+m^{2} & \mathrm{if} \; n \geq 3. 
		 \end{array}
	     \right.
\end{equation*}
On the other hand, if $n\leq 0$ we get that $\oo_{\ell}(nH)$, $\oo_{X}(nH-L)$ and $\oo_{X}(nH-3L)$ are not effective thus $h^{0}(\oo_{\ell}(nH))=h^{0}(\oo_{X}(nH-L))=h^{0}(\oo_{X}(nH-3L))=0$. Therefore, by the associated exact sequence in cohomology of \eqref{exacseqK} we get that:
\begin{equation*}
    h^{0}(C,\oo_{C}(n))=h^{0}(X,\mathcal{K}(n))=h^{1}(X,\oo_{X}(nH-3L).
\end{equation*}
\end{proof}

Therefore, all that remains is to calculate $h^{1}(X,\oo_{X}(nH-mL))$ for all $n\leq 0$. For this, we distinguish two cases:

\begin{lem} \label{lemacohomologiams}
    Let $m\geq 3$ and $n\in{\{-2m+1, -m, \ldots -2, -1, 0\}} $ then we have that:
    $$h^{1}(X,\oo_{X}(nH-mL))=\left\{
  \begin{array}{ll}
       \left\lfloor \frac{2m+n}{ 2}\right\rfloor\left(2m+n-\left\lfloor \frac{2m+n}{ 2}\right\rfloor\right) & \mathrm{if} \; n<0; \\
       (m-1)(m+1) & \mathrm{if} \; n=0.\\ 
         \end{array}
    \right.
    $$
\end{lem}
\begin{proof}
    Let $n\in\{-2m+1 , -m, \ldots , -2, -1, 0\}$ and $i\in\{1,. …\ldots , m-1\}$, then we have the following exact sequence
    \begin{equation} .\label{eqexacseq1}
        0\to \oo_{X}(nH-(m-i+1)L)\to \oo_{X}(nH-(m-i)L)\to \oo_{L}(nH-(m-i)L)\to 0
    \end{equation}
 Since $nH-(m-i)L$ is not a effective divisor on $X$, we have that $H^0(X,\oo_{X}(nH-(m-i)L))=0$. On the other hand, if $i\in\{1, \ldots, m-2\}$, we have that $-2-deg_{L}(nH-(m-i)L)\leq 0$ then $H^{1}(L,\oo_{L}(nH-(m-i)L))=0$  and for $i=m-1$ we have $H^{1}(X,\oo_{X}(nH-(m-i)L))=H^{1}(X,\oo_{X}(nH-L))=0$ since $L$ is an ACM curve. Therefore we have from the long exact sequence on cohomology associate to the exact sequence \eqref{eqexacseq1}:
 \begin{align*}
     h^{1}(X,\oo_{X}(nH-(m-i+1)L))=h^{1}(X,\oo_{X}(nH-(m-i)L))+h^{0}(L,\oo_{L}(nH-(m-i)L))
 \end{align*}
 If $n<0$, then $deg_{L}(nH-(m-i)L)=n+2m-2i\geq 0$ only if $i\leq \left\lfloor \frac{2m+n}{ 2}\right\rfloor$, thus we have that
 \begin{align*}
     h^{1}(X,\oo_{X}(nH-mL))&=h^{1}(X,\oo_{X}(nH-(m-1)L))+h^{0}(L,\oo_{L}(nH-(m-1)L) )\\
     &\qquad \vdots \\
     &=h^{1}(X,\oo_{X}(nH-L))+\sum_{i=1}^{\left\lfloor \frac{2m+n}{ 2}\right\rfloor}h^{0}(L,\oo_{L}(nH-(m-i)L)) \\
     &=\sum_{i=!}^{\left\lfloor \frac{2m+n}{ 2}\right\rfloor}h^{1}\oo_{\PP^{2}}(n+2(m-i))\\
     &=\sum_{i=!}^{\left\lfloor \frac{2m+n}{ 2}\right\rfloor}((n+2(m-i))+1) \\
     &=\left\lfloor \frac{2m+n}{ 2}\right\rfloor(n+2m+1)-\left\lfloor \frac{2m+n}{ 2}\right\rfloor\left(\left\lfloor \frac{2m+n}{ 2}\right\rfloor +1\right)\\
     &=\left\lfloor \frac{2m+n}{ 2}\right\rfloor(n+2m-\left\lfloor \frac{2m+n}{ 2}\right\rfloor).
 \end{align*}

 If $n=0$, then $deg_{L}(nH-(m-i)L)=2m-2i\geq 0$ only if $i\leq m$ thus we have that
\begin{align*}
    h^{1}(X,\oo_{X}(-mL))&=h^{1}(X,\oo_{X}(-L))+\sum_{i=1}^{m-1}h^{0}(L,\oo_{L}(2(m-i)) )\\
    &=\sum_{i=1}^{m-1}(2(m-i)+1)\\
    &=(m-1)(m+1).
\end{align*}
\end{proof}

\begin{lem} \label{lema0coh}
    Let $m\geq 1$ and $n\leq -m-2$, then we have that:
    $$h^{1}(X,\oo_{X}(nH-mL)=0.$$
\end{lem}
\begin{proof}
First, for $k\geq 1$, we consider the exact sequence
\begin{equation} \label{exacseq2}
0\to \oo_{X}((m+k)H+mL)\to \oo_{X}((m+k+1)H+mL)\to \oo_{H}((m+k+1)H+mL)\to 0.    
\end{equation}
Since $deg_{H}((m+k+1)H+mL)=5m+4k+4>4$ this divisor is no special, then $h^{1}(H,\oo_{H}((m+k+1)H+mL))=0$. Then we obtain from the long exact sequence associated to \eqref{exacseq2} that
$$h^{1}(X,\oo_{X}((m+k+1)H+mL))\leq h^{1}(X,\oo_{X}((m+k)H+mL)) \quad \text{for all}\quad k\geq 1.$$
By Corollary \ref{cor:irreduciblecurves} and \cite{k3sup}*{Rmk. 2.2} we have that
 $$h^{1}(X,\oo_{X}((2m)H+mL))=0$$ then $h^{1}(X,\oo_{X}((2m+k)H+mL))=0$ for all $k\geq 1$ and therefore the result is following by Serre duality.\\
\end{proof}

Finally by Remark \ref{rmk:igualdad} we can compute the Rao function of a curve $F_{m}$ for $m\geq 3$:
\begin{cor} Let $C=F_{m}\subseteq X$, then:
    \begin{itemize}
        \item If $m=3$, then the Rao funtion of $C$ is:
         $$\rho_{C}(n)=\left\{
  \begin{array}{ll}
         0                &  \mathrm{if} \; n\leq -5\\
         1                &  \mathrm{if} \; n=-4 \\
         2                &  \mathrm{if} \; n=-3 \\
         4                &  \mathrm{if} \; n=-2 \\
         6                &  \mathrm{if} \; n=-1 \\
         8                &  \mathrm{if} \; n=0 \\
         8                &  \mathrm{if} \; n=1 \\
         5                &  \mathrm{if} \; n=2\\
         2                &  \mathrm{if} \; n=3 \\
         0   & \mathrm{if} \; n \geq 4. 
		 \end{array}
	     \right.$$

         \item If $m=4$, then the Rao funtion of $C$ is:
          $$\rho_{C}(n)=\left\{
  \begin{array}{ll}
         0                &  \mathrm{if} \; n\leq -6\\
         2                &  \mathrm{if} \; n=-5 \\
         4                &  \mathrm{if} \; n=-4 \\
         6                &  \mathrm{if} \; n=-3 \\
         9                &  \mathrm{if} \; n=-2 \\
         12               &  \mathrm{if} \; n=-1 \\
         15               &  \mathrm{if} \; n=0 \\
         16               &  \mathrm{if} \; n=1 \\
         14               &  \mathrm{if} \; n=2 \\
         8                &  \mathrm{if} \; n=3 \\
         3                &  \mathrm{if} \; n=4 \\
         0   & \mathrm{if} \; n \geq 5. 
		 \end{array}
	     \right.$$
         
         \item If $m\geq 5$, then the Rao funtion of $C$ is:
         $$\rho_{C}(n)=\left\{
  \begin{array}{ll}
         0                &  \mathrm{if} \; n\leq -m-2\\
         \left\lfloor \frac{2m+n}{ 2}\right\rfloor\left(2m+n-\left\lfloor \frac{2m+n}{ 2}\right\rfloor\right)                &  \mathrm{if} \; -m-1 \leq n<0 \\
         (m-1)(m+1)       &  \mathrm{if} \; n=0 \\
         m^{2}+mn-2n^{2}-2        &  \mathrm{if} \; 0<n\leq m-1 \\
        m-1      &  \mathrm{if} \; n=m \\
         0   & \mathrm{if} \; n \geq m+1. 
		 \end{array}
	     \right.$$
    \end{itemize}
\end{cor}
Now we are able to compute the Rao function of a curve in any class $(k+m)H-mL$ or $kH+mL$.
\begin{cor}
    Let $C$ be a curve in the linear system $|(k+m)H-mL|$, (resp. $|kH+mL|$) then
    $$\rho_{C}(n)=\rho_{F_{m}}(m+k-n) \qquad \forall n \in \ZZ.$$
\end{cor}
\begin{proof}
    Let $C\in|(k+m)H-mL|$, (resp. $C\in{|kH+mL|})$. Since $C$ is obtained by elementary biliaison $(4,1)$ from a curve $D_{k-1}$ in the linear system $|(k-1+m)H-mL|$
    (resp. $|(k-1)H+mL|$), inductively we conclude by \cite{mdp2}*{Prop. 3.3} that $\rho_{C}(n)=\rho_{D_{1}}(n+1-k)$.
    On the other hand, $D_{1}$ is linked to $F_{m}$ (resp. $D_{1}=G_{m}$ and is linked to $F_{m}$ by the complete intersection of $X$ and a surface of degree $m+1$) by the complete intersection of $X$ with a surface of degree $m+1$, then by \cite{mdp2}*{Prop. 1.2} we have that,
  in both cases,  
  $$\rho_{D_{1}}(n)=\rho_{F_{m}}(m+1-n)$$
  therefore using both formulas we obtain the result.\\
\end{proof}

\appendix

\section{Macaulay2 code}

First we establish a base code, which goes at the beginning of any code in this section. This code give us a surface $X$ in $\PPP$ that contain a line $L$

\begin{verbatim}
    R=ZZ/2011[x,y,z,w]
------R=QQ[x,y,z,w]
 randomElement = (d, I) ->
{
    randomElementR := I;
    randomElementI := I;
    if(toString class I == "PolynomialRing") then
    {
        randomElementR = I;
        randomElementI = ideal vars randomElementR;
    }
    else if(toString class I == "Ideal") then
    {
        randomElementR = ring I;
        randomElementI = I;
    };
    randomElementF := sub(0, randomElementR);
for p in flatten entries gens randomElementI do 
(randomElementF = randomElementF + p * (random(randomElementR^{d- (degree p)_0},
randomElementR^{0}))_0_0);
    return randomElementF;
}
l11=randomElement(1,ideal(vars R));
l12=randomElement(1,ideal(vars R));
L=ideal(l11,l12);
a=randomElement(4,L);
X=ideal(a);

\end{verbatim}

To build curves in the class $(k+m)H-mL$, first we need to build a curve $F_{m}$. (remember always replace m with a non-zero natural number). With this curve
the elements in the class $(k+m)H-mL$ are denoted by $Ck$ (remember always replace m and k with a non-zero natural number)

\begin{verbatim}
    Fm=saturate((L^m)+X);
    
    b=randomElement(k+m,Fm);
    Ck=saturate(ideal(a,b),Fm);
\end{verbatim}

To build curves in the class $kH+mL$, first we need to build a curve $G_{m}$. (remember always replace m with a non-zero natural number). With this curve
the elements in the class $kH+mL$ are denoted by $Dk$ (remember always replace m and k with a non-zero natural number)
\begin{verbatim}
    for i in 0..(m-1) do h_i=randomElement(1,L);
    for i in 0..(m-1) do B_i=saturate(ideal(h_i)+X,L);
    Gm=B_0;
    for i in 1..(m-1) do Gm = saturate(Gm*B_i);

    b=randomElement(k+m,Gm);
    Dk=saturate(ideal(a,b),Gm);
\end{verbatim}

Finally, to compute the rank of $H^{0}(C,N_{C})$ and $H^{1}(C,N_{C})$ for a curve $C$ we use the following code:
\begin{verbatim}
    N = sheaf Hom( C / C^2, R^1 / C );----->the normal sheaf of C on P3
    rank HH^0 N
    rank HH^1 N
\end{verbatim}

\end{document}